\documentclass[a4paper,11pt]{article}

\usepackage[top=3.0cm, bottom=3.0cm, inner=3.0cm, outer=3.0cm,
includefoot]{geometry}

\usepackage{verbatim}
\usepackage{caption}
\usepackage{url}
\usepackage{amsmath}
\usepackage{geometry}
\usepackage{amssymb}
\usepackage{amsmath}
\usepackage{graphicx}
\usepackage{amsthm}
\usepackage{bbm}
\usepackage{float}
\usepackage{color,soul}
\usepackage{hyperref}
\usepackage[T1]{fontenc}
\usepackage[utf8]{inputenc}
\usepackage{authblk}
\usepackage{booktabs}
\usepackage{longtable}
\usepackage{enumerate}
\usepackage{tikz}
\usetikzlibrary{arrows}
\usetikzlibrary{decorations.markings}
\usepackage{standalone}
\usepackage{indentfirst}
\usepackage{mathtools}

\newcommand{\eChar}{\begin{enumerate}[(i)]}
\newcommand{\eCharR}{\begin{enumerate}[(a)]}
\newcommand{\eBr}{\begin{enumerate}[(1)]}

\newcommand\floor[1]{\left\lfloor #1\right\rfloor}
\newcommand\ceil[1]{\left\lceil #1\right\rceil}

\title
{An extremal theorem for positive curvature of graphs}

\author[1]{Kaizhe Chen\thanks{Email: ckz22000259@mail.ustc.edu.cn}}
\author[2]{Shiping Liu\thanks{Email: spliu@ustc.edu.cn}}
\author[3]{Zhe You\thanks{Email: y30231280@mail.ecust.edu.cn}}
\affil[1]{School of the Gifted Young, University of Science and Technology of China}
\affil[2]{School of Mathematical Sciences, University of Science and Technology of China}
\affil[3]{School of Mathematics, East China University of Science and Technology}
\date{}

\theoremstyle{plain}
\newtheorem{lemma}{Lemma}[section]
  {%
   \endlemma
   \endgroup}
   
\newtheorem{theorem}[lemma]{Theorem}

\theoremstyle{definition}

\newtheorem{claim}{Claim}[lemma]

\newtheorem{definition}[lemma]{Definition}

\newtheorem{problem}[lemma]{Problem}

\newtheorem{remark}[lemma]{Remark}

\numberwithin{equation}{section}

\begin{document}
\maketitle

\begin{abstract}
We prove an extremal theorem  for positive Ollivier/Lin--Lu--Yau curvature: every graph of order \(n\geq 8\) with more than
\[
  T(n)=\frac{n^2-3n}{2}-\left\lceil\frac{n}{2}\right\rceil+2
\]
edges has positive Ollivier/Lin--Lu--Yau curvature, and this threshold is optimal. 
Moreover, for even $n\geq 12$, there exists a unique graph with $T(n)$ edges that has an edge with non-positive curvature.
For $n=8,10$ and odd $n\geq 9$, the extremal graphs are not unique.
This suggests a new class of extremal graph-theoretic problems arising from discrete curvature notions.
\end{abstract}



\section{Introduction}
Determining the maximum or minimum number of edges in a graph under prescribed
restrictions is a central theme in extremal graph theory. This perspective has led to many fundamental insights into discrete structures~\cite{Bollobas, Simonovits}. 
Classical examples include Tur\'an-type problems, where the
restriction is given by forbidding certain subgraphs. 
More recently, spectral
extremal problems, in which the restrictions are imposed on eigenvalues of graphs, have also attracted considerable attention~\cite{spectral-extremal,Nikiforov}.

Graphs satisfying discrete curvature restrictions have been studied extensively.
In particular, the combinatorial structure of graphs with lower Ricci curvature bounds has attracted much attention in recent years; see \cite{NR17} and the references therein. Roughly speaking, positive or large discrete curvature tends to force a graph to have many edges.

In this paper, we establish an extremal result for a discrete
Ricci curvature, namely the Ollivier/Lin--Lu--Yau curvature. The
Lin--Lu--Yau curvature~\cite{LLY11} is a modified version of Ollivier's Ricci
curvature~\cite{O09} on graphs. 
For an edge \(xy\), let \(\mu_x^p\) and \(\mu_y^p\) be the \(p\)-lazy random walk measures at \(x\) and \(y\), respectively. 
The corresponding \(p\)-Ollivier curvature is
\[
  \kappa_p(x,y)
  :=
  1-W(\mu_x^p,\mu_y^p),
\]
where \(W\) denotes the $L^1$-Wasserstein distance. 
The Lin--Lu--Yau curvature
\(\kappa_{\mathrm{LLY}}(x,y)\) is obtained from the behavior of
\(\kappa_p(x,y)\) as \(p\to 1\). It is known that $\kappa_{LLY}(x,y)$ is positive if and only if $\kappa_p(x,y)$ is positive for some $p$ close to $1$. In particular we have 
\[
  \kappa_{\mathrm{LLY}}(x,y)>0
  \quad
  \Longleftrightarrow
  \quad
  \kappa_{1/2}(x,y)>0;
\]
see \cite{BCLMP18} or \eqref{eq:bourne} below for more precise interval of $p$. For connectivity, diameter, eigenvalues, and other combinatorial properties of graphs
with positive Ollivier/Lin--Lu--Yau curvature, see
\cite{CLY,CKKLMP20,Hehl-regular,MW19} and the references therein.

A graph is said to have positive Lin--Lu--Yau curvature if
\(\kappa_{\mathrm{LLY}}(x,y)>0\) for every edge \(xy\). We now state our curvature extremal result.
\begin{theorem}\label{main}
 For any integer $n$, set 
 \begin{equation*}
    T(n)=\frac{n^2-3n}{2}-\left\lceil\frac{n}{2}\right\rceil+2.
\end{equation*}
Then for any $n\geq 8$, any graph with $n$ vertices and more than $T(n)$ edges has positive Lin--Lu--Yau curvature.
Moreover, for even $n\ge 12$, there exists a unique graph with $n$ vertices and $T(n)$ edges that has an edge with non-positive Lin--Lu--Yau curvature.
\end{theorem}

We remark that the assumption \(n\geq 8\) is necessary. Moreover, the extremal graph is not unique for $n=8,10$ and for all odd $n\ge 9$, as shown in Section~\ref{section:sharpness}. 

The condition for positive curvature in Theorem~\ref{main} can be restated in terms of the number of edges in the complement. More precisely, for a graph $G$ of order \(n\geq 8\), if its complement \(\overline{G}\) has at most
\[
  n+\left\lceil\frac{n}{2}\right\rceil-3
\]
edges, then \(G\) has positive Lin--Lu--Yau curvature.
A closely related criterion for positive Lin--Lu--Yau curvature, formulated in terms of forbidden subgraphs in the complement, was obtained in~\cite{CLY26}.

Theorem~\ref{main} suggests a new class of extremal graph-theoretic
problems arising from discrete Ricci curvature notions; see
Section~\ref{section:extremal-problems} for further discussion.

We outline the strategy of the proof of Theorem~\ref{main}. Suppose, to the contrary, that some edge \(xy\) has non-positive Lin--Lu--Yau curvature. We first select a special transport plan described in Lemma~\ref{maximizer}. The extremality of this plan forces forbidden adjacencies between certain vertex classes in the neighborhoods of $x$ and $y$. The non-positivity of curvature then yields lower bounds on the sizes of these classes. Combining this with the assumed bound on the number of edges reduces the desired contradiction to a numerical inequality proved in Lemma~\ref{computation}.

While this paper was being written, we became aware of a closely related recent article \cite{Yamada}. In that article, the authors proved that the graphs obtained from the complete graph by deleting a matching, a star, or a cycle have non-positive $0$-Ollivier curvature. They further proposed the question of determining which edges, and how many, must be removed from a complete graph so that some edges attain non-positive Ollivier curvature. 

Throughout the paper, we use the following notation. 
Let $G=(V,E)$ be a simple finite graph.
Denote by $\overline{G}$ the complement graph of $G$.
For any $x\in V$, let $N(x)$ be the set of neighbors of $x$ and let $d_x:=|N(x)|$ be its degree. 
For any two vertices $x$ and $y$, we denote the distance between them by $\rho(x,y)$. 
Let $\delta_{xy}: V\times V\to [0,1]$ denote the function defined as follows:
    \[\delta_{xy}(a,b)=\left\{
                \begin{array}{ll}
                    1, & \hbox{if $a=x$, $b=y$;}\\
                    0, & \hbox{otherwise.}
                \end{array}
                \right.
    \]
For two fixed vertices $x$ and $y$, we set $A_{xy}\coloneqq N(x)\cap N(y)$, $N^{(xy)}_x\coloneqq N(x)\backslash (A_{xy}\cup \{ y \})$, and $N^{(xy)}_y\coloneqq N(y)\backslash (A_{xy}\cup \{ x \})$.
For simplicity, we write $N_x$ for $N^{(xy)}_x$ and write $N_y$ for $N^{(xy)}_y$ hereafter.

\section{Preliminaries}
Before introducing the Lin--Lu--Yau curvature, we first recall the definition of the Wasserstein distance.
\begin{definition}[$L^1$-Wasserstein distance]
     Let $G=(V,E)$ be a locally finite graph, $\mu_1$ and $\mu_2$ be two probability measures on $G$. 
     The {\it $L^1$-Wasserstein distance} $W(\mu_1, \mu_2)$ between $\mu_1$ and $\mu_2$ is defined as
     \begin{align}\label{defi}
         W(\mu_1,\mu_2)=\inf_{\pi}\sum_{u\in V}\sum_{v\in V}\rho(u,v)\pi(u,v),
     \end{align}
     where the infimum is taken over all the mappings $\pi: V\times V\to [0,1]$ satisfying
     $$\mu_1(u)=\sum\limits_{v\in V}\pi(u,v) \text{ for any}\ u\in V$$
     and
     $$\mu_2(v)=\sum\limits_{u\in V}\pi(u,v) \text{ for any}\ v\in V.$$ 
     Such a mapping is called a {\it transport plan} from $\mu_1$ to $\mu_2$. 
     A transport plan that attains the infimum in \eqref{defi} is called {\it optimal}.
     We call a transport plan {\it simple} if, for each vertex $u\in V$, we have
     $$\pi(u,u)=\min\{ \mu_1(u), \mu_2(u)\}.$$
\end{definition}
 Here, for a given idleness parameter $p\in [0,1]$, we consider the particular measure $\mu_x^p$ around a vertex $x\in V$ defined as follows:
    \[\mu_x^p(y)=\left\{
                    \begin{array}{ll}
                      p, & \hbox{if $y=x$;} \\
                      \frac{1-p}{d_x}, & \hbox{if $xy\in E$;} \\
                      0, & \hbox{otherwise.}
                    \end{array}
                  \right.
     \]
     
   Based on the probability measure above, two kinds of Ricci curvature on graphs are defined as follows.
\begin{definition}[$p$-Ollivier curvature and Lin--Lu--Yau curvature] 
Let $G=(V,E)$ be a locally finite graph. 
For any two distinct vertices $x,y$ in $G$, the {\it $p$-Ollivier curvature} $\kappa_p(x,y)$, $p\in [0,1]$, is defined as
     \[\kappa_p(x,y)=1-\frac{W(\mu_x^p,\mu_y^p)}{\rho(x,y)}.\]
The {\it Lin--Lu--Yau curvature} $\kappa_{LLY}(x,y)$ is defined as
     \[\kappa_{LLY}(x,y)=\lim_{p\to 1}\frac{\kappa_p(x,y)}{1-p}.\]
\end{definition}
It is worth noting that the ratio $\kappa_p(x,y)/(1-p)$ is constant when $p$ is large enough. 
In fact, it was proved in \cite{BCLMP18} that
\begin{equation}\label{eq:bourne}
  \kappa_{LLY}(x,y)=\frac{\kappa_p(x,y)}{1-p} \,\,\text{for any $p \in \left[\frac{1}{\max\{d_x,d_y\}+1},1\right]$}.
\end{equation}

For any locally finite graph $G$, the  normalized graph Laplacian $\Delta$ is defined as $$\Delta f(x):=\frac{1}{d_x} \sum_{y: xy\in E(G)}(f(y)-f(x)), \text{ for any $f: V(G)\to \mathbb{R}$ and any $x\in V(G)$}.$$
 There is another limit-free formulation of Lin--Lu--Yau curvature, which was given by M\"{u}nch and Wojciechowski \cite{MW19}. 
 \begin{theorem}[Curvature via Laplacian {\cite[Corollary 2.2]{MW19}}]\label{Curvature via the Laplacian}
     Let $G$ be a locally finite graph and let $xy$ be an edge. Then
     $$\kappa_{LLY}(x, y)=\inf _{\substack{f:N(x)\cup N(y)\to \mathbb{Z}\\f \in Lip(1) \\ f(y)-f(x)=1}} \left(\Delta f(x)-\Delta f(y)\right).$$
 \end{theorem}

The following lemma is convenient for our proof.

\begin{lemma}[{\cite[Corollary 3.6]{Hehl-regular}}]\label{triangle}
    Let $G=(V, E)$ be a locally finite graph. 
    Let $x, y \in V$ be two adjacent vertices of equal degree $d$. 
    Then $\kappa_{LLY}(x, y) \in \mathbb{Z} / d$. 
    Furthermore,
$$
-2+\frac{4}{d}+ \frac{3|A_{x y}|}{d} \leq \kappa_{LLY}(x, y) \leq \frac{2+|A_{xy}|}{d} .
$$
\end{lemma}

The following lemma is useful in calculating the Lin--Lu--Yau curvature.

\begin{lemma}[{\cite[Lemma 4.1]{CLY}}]\label{simple}
    Let $G$ be a graph. 
    Let $x$ and $y$ be two adjacent vertices in $G$ with $d_x\ge d_y$. 
    For any $p\in \left[ \frac{1}{1+d_y},1 \right]$, there is a simple optimal transport plan $\pi$ from $\mu_x^p$ to $\mu_y^p$ such that $$\pi(x,y)=p-\frac{1-p}{d_y}.$$ 
\end{lemma}

Based on the above Lemma, we derive the following result.




\begin{lemma}\label{maximizer}
    Let $G$ be a graph. 
    Let $x$ and $y$ be two adjacent vertices in $G$ with $d_x\ge d_y$. 
    For any $p\in \left[ \frac{1}{1+d_y},1 \right]$, there exists a simple optimal transport plan from $\mu_x^p$ to $\mu_y^p$ such that $\pi(x,y)=p-\frac{1-p}{d_y}$ and that, among all such plans, maximizes $\sum_{\rho(u,v)=1}\pi(u,v)$.
\end{lemma}

\begin{proof}
Let $ X=\operatorname{supp}\mu_x^p$ and  $Y=\operatorname{supp}\mu_y^p$.
Then $X=\{x\}\cup N(x)$ and $Y=\{y\}\cup N(y)$. 
Since $G$ is locally finite, both $X$ and $Y$ are finite. 
Hence every transport plan from $\mu_x^p$ to $\mu_y^p$ can be regarded as a matrix in $\mathbb{R}^{X\times Y}$.
The set of all transport plans is determined by finitely many linear equalities. 
Moreover, every entry of the corresponding matrix lies in $[0,1]$. 
Hence this set is compact.

Let $\Pi_{xy}^p$ be the set of all simple optimal transport plans from 
$\mu_x^p$ to $\mu_y^p$ such that $\pi(x,y)=p-\frac{1-p}{d_y}$.
By Lemma~\ref{simple}, the set $\Pi_{xy}^{p}$ is non-empty. 
The additional condition that the plan is optimal is closed. Indeed, the cost functional
$\mathcal C(\pi)
  =
  \sum_{u\in X}\sum_{v\in Y}\rho(u,v)\pi(u,v)
$
is continuous.
Since the set of all transport plans is compact, the minimum of
$\mathcal C$ is attained, and the set of optimal plans is the closed level set
where $\mathcal C$ equals this minimum.
Moreover, the condition 
$$
\pi(x,y)=p-\frac{1-p}{d_y}
$$
is closed. The simplicity condition is also closed, since it is given by the finite family of linear equalities
$\pi(u,u)=\min\{\mu_x^p(u),\mu_y^p(u)\}, \text{for all } u\in X\cap Y.$
Thus $\Pi_{xy}^{p}$ is a closed subset of a compact set, and hence is compact.

Now define
$$
\Phi(\pi)=\sum_{\rho(u,v)=1}\pi(u,v).
$$
Then $\Phi$ is a continuous function on $\Pi_{xy}^{p}$.
Therefore, by the Weierstrass extreme value theorem, $\Phi$ attains its maximum on $\Pi_{xy}^{p}$. 
Consequently, there exists some $\pi_0\in \Pi_{xy}^{p}$ such that
$$ \Phi(\pi_0)=\max_{\pi\in\Pi_{xy}^{p}}\Phi(\pi).$$
This proves the lemma.
\end{proof}


\section{A computational lemma}
In this section, we present a computational lemma which plays a key role in our proof.
\begin{lemma}\label{computation}
    Let \(d_x,d_y,N_x,N_y,a,n,V_1,V_2,V_3,V_4,V_5\) be non-negative integers and let $M$ be a non-negative real number. 
Suppose that $n\geq 8$, $d_x=1+a+N_x$, $d_y=1+a+N_y$, $N_x\geq N_y$, $N_x+N_y+a+2\leq n$, and $V_1+V_3+V_4=N_x$.
Suppose that, if $N_y=a=0$, then $V_4=0$.
Assume further that 
$$d_x-\delta d_y>0,\quad V_1\geq 1+a+\frac{d_x}{d_y},\quad V_2\geq \frac{d_y(a+2)}{d_x-\delta d_y}+M,\quad V_5\geq \frac{d_yV_4}{d_x-\delta d_y}-M,$$
where $\delta \coloneqq 0$ if $N_y\ge 1$, and $\delta \coloneqq 1$ if $N_y= 0$. Then, we have
    \begin{align}\label{inequality}
      \floor{\frac n2}-N_x-N_y-2a+V_2(V_1+V_3)+V_5V_1\geq 2.  
    \end{align}
    Moreover, for even $n$, equality in \eqref{inequality} holds if and only if $$V_1=\left\lceil\frac{n}{2} \right\rceil,a=0,V_2=N_y=1,V_4=\left\lfloor\frac{n}{2}\right\rfloor-3,\ \ {\rm and}\ \ V_3=V_5=0.$$
\end{lemma}

\begin{proof}
For simplicity, set $p\coloneqq \ceil{\frac{d_y(a+2)}{d_x-\delta d_y}}$, and $q\coloneqq \ceil{\frac{d_y(a+2+V_4)}{d_x-\delta d_y}}$.
Since $V_2$ and $V_5$ are integers and \(M\geq 0\), we find $V_2\geq p$ and $V_2+V_5\geq q$.
Therefore,
\begin{align}\notag
V_2(V_1+V_3)+V_5V_1&= qV_1+pV_3+(V_2-p\bigr)V_3+(V_2+V_5-q)V_1 \\ \label{v2v1v3}
&\geq qV_1+pV_3.
\end{align} 
Note that $q\ge p\geq 1$. So, $pV_3 \ge V_3= N_x-V_1-V_4$.
It follows from \eqref{v2v1v3} that
\begin{align*}
\floor{\frac n2}-N_x-N_y-2a+V_2(V_1+V_3)+V_5V_1&\ge
\floor{\frac n2}-N_x-N_y-2a+qV_1+N_x-V_1-V_4 \\
&=\floor{\frac n2}-N_y-2a-V_4+(q-1)V_1.
\end{align*}
Now, it suffices to show that $\floor{\frac n2}-N_y-2a-V_4+(q-1)V_1\geq 2$.

Recall that $n\ge 8$ and $q\geq 1$. 
By the assumption, if $N_y=a=0$, then $V_4=0$, and hence
\begin{equation}\label{aNy0}
    \floor{\frac n2}-N_y-2a-V_4+(q-1)V_1> 2.
\end{equation}
Thus, we may assume that $a+N_y>0$.

We split the proof into two cases according to whether $q=1$.

\medskip

\noindent
\textbf{Case 1: \(q=1\).}

We have $\ceil{\frac{d_y(a+2+V_4)}{d_x-\delta d_y}}=1$, and hence $d_y(a+2+V_4+\delta)\leq d_x$.
Since \(d_x=1+a+N_x\) and \(d_y=1+a+N_y\), we obtain $(1+a+N_y)(a+2+V_4+\delta)\leq 1+a+N_x$.
That is, 
\begin{equation}\label{NxaNy}
    N_x\geq(1+a+N_y)(a+2+V_4+\delta)-a-1.
\end{equation}
Recall that $a+N_y>0$. So,
\begin{align}\label{q=1}
&(1+a+N_y)(a+2+V_4+\delta)-a-1-\bigl(N_y+3a+2V_4+2\bigr)\nonumber\\
=&(a+N_y-1)V_4+N_y\delta +a^2+(\delta+N_y-1)(a+1)\ge 0.
\end{align}
Therefore, $N_x\geq N_y+3a+2V_4+2$.
Together with the assumption $N_x+N_y+a+2\leq n$, we get
$n\geq 2N_y+4a+2V_4+4$, and hence
$\floor{\frac n2}\geq N_y+2a+V_4+2.$
Since \(q=1\), we conclude that
\begin{align}\label{q=1,last}
    \floor{\frac {n}{2}}-N_y-2a-V_4+(q-1)V_1
=\floor{\frac{n}{2}}-N_y-2a-V_4\geq 2.
\end{align}

\medskip

\noindent
\textbf{Case 2: $q\geq 2$.}

Since $V_1\geq 1+a+\frac{d_x}{d_y}$ and $V_1$ is an integer, we have $V_1\geq 1+a+\ceil{\frac{d_x}{d_y}}$.
Let $\theta\coloneqq \frac{d_x}{d_y}$.
Then $V_1\geq 1+a+\ceil{\theta}$.
Since $n\geq N_x+N_y+a+2$, we have
$$\floor{\frac {n}{2}}-N_y-2a\geq \floor{\frac{N_x-N_y-3a+2}{2}}.$$
Therefore,
\begin{align*}
\floor{\frac{n}{2}}-N_y-2a-V_4+(q-1)V_1 
\ge\floor{\frac{N_x-N_y-3a+2}{2}}-V_4+(q-1)(1+a+\ceil{\theta})
\end{align*}
By the definition of $q$, we have $\frac{d_y(a+2+V_4)}{d_x-\delta d_y}\leq q$, which implies $V_4\leq q\theta-a-2$. So,
\begin{align*}
&\floor{\frac{n}{2}}-N_y-2a-V_4+(q-1)V_1 \\
\geq&\floor{\frac{N_x-N_y-3a+2}{2}}-\left(q\theta-a-2\right)+(q-1)(1+a+\ceil{\theta}) \\
=&\floor{\frac{N_x-N_y+a+8}{2}}+(q-2)(1+a)+(q-1)(\ceil{\theta}-\theta)-\theta.
\end{align*}
By the assumption that $q\ge 2$ and the fact that $\ceil{\theta}\ge\theta$, we arrive at
\begin{align*}
\floor{\frac{n}{2}}-N_y-2a-V_4+(q-1)V_1 
&\geq\floor{\frac{N_x-N_y+a+8}{2}}-\theta\\
&> \frac{N_x-N_y+a+6}{2}-\theta.
\end{align*}
By the assumptions \(d_x=1+a+N_x\), \(d_y=1+a+N_y\), $N_x\ge N_y$, and $a+N_y>0$, we have
\begin{align*}
\theta=\frac{d_x}{d_y}=1+\frac{N_x-N_y}{1+a+N_y}\le 
1+\frac{N_x-N_y}{2}.
\end{align*}
Thus,
\begin{align*}
\floor{\frac{n}{2}}-N_y-2a-V_4+(q-1)V_1 
&>
\frac{N_x-N_y+a+6}{2}-1-\frac{N_x-N_y}{2}\ge 2.
\end{align*}

\noindent\textbf{Characterization of the equality for even $n$}

Recall from \eqref{aNy0} that $a+N_y> 0$.
Note that the equality in~\eqref{inequality} cannot hold in Case 2. So, $q=1$, and hence the equality in~\eqref{q=1} holds, which implies $a=0$, $N_y=1$, and $\delta =0$.
Moreover, the equality in~\eqref{q=1,last} should also hold.
So, $V_4=\floor{\frac{n}{2}}-3$.
Note that the equality in \eqref{NxaNy} gives $N_x=n-3$.
Thus, $V_1\geq 1+a+\frac{d_x}{d_y}=\frac{n}{2}$.
Then, $V_1+V_3+V_4=N_x$ implies $V_1=\frac{n}{2}$ and $V_3=0$.
The equality in~\eqref{v2v1v3} gives $V_2=p$ and $V_2+V_5= 1$.
Recall that $q\ge p\ge 1$. So $p=1$, yielding $V_2=1$ and $V_5=0$.
\end{proof}

\section{\texorpdfstring{Proof of the main theorem}{Proof of the main theorem}}

In this section, we prove Theorem~\ref{main}.
Let $n\ge 8$ and let $G$ be a graph of order $n$. 
Suppose that there are at most $n+\left\lceil \frac{n}{2} \right\rceil -3+\epsilon_n$ edges in the complement graph $\overline{G}$ of $G$, where $\epsilon_n =1$ if $n$ is even and $n\ge 12$, and $\epsilon_n =0$ otherwise.
Assume that there is an edge $xy$ in $G$ such that the Lin--Lu--Yau curvature $\kappa_{LLY}(x,y)$ is non-positive. 
We will prove that this forces $\epsilon_n =1$ and that $G$ is uniquely determined.

 Without loss of generality, assume that $d_x\ge d_y$.
  For any fixed $p\in \left[ \frac{1}{1+d_y},1 \right)$,  by Lemma~\ref{maximizer}, there exists a simple optimal transport plan from $\mu_x^p$ to $\mu_y^p$ such that $\pi(x,y)=p-\frac{1-p}{d_y}$ and that, among all such plans, maximizes $\sum_{\rho(u,v)=1}\pi(u,v)$.
A direct observation is that, for any two vertices $u,v$ with $\pi(u,v)>0$, we have either $u=x$ and $v=y$, or $u\in N_x$ and $v\in N_y\cup A_{xy}$.
 Set $$V_1\coloneqq \{ w\in N_x|\ \exists z,\ \pi(w,z)>0,\ \rho(w,z)\geq 2 \}$$ and $$V_2\coloneqq \{ w\in A_{xy} \cup N_y|\ \exists z,\ \pi(z,w)>0,\ \rho(z,w)\geq 2 \}.$$
 
\begin{lemma}\label{E(V_1,V_2)emptyset}
    $E(V_1,V_2)=\emptyset$.
\end{lemma}
\begin{proof}
    For a contradiction, assume that there exist $u_1\in V_1$ and $v_1\in V_2$ such that $u_1v_1\in E(V_1,V_2)$.
    By the definition of $V_1$, there exists a vertex $v_2\in V_2$ such that $\pi(u_1,v_2)>0$ and $\rho(u_1,v_2)\geq 2$.
     By the definition of $V_2$, there exists a vertex $u_2\in V_1$ such that $\pi(u_2,v_1)>0$ and $\rho(u_2,v_1)\geq 2$.
Set $\pi'\coloneqq\pi-\varepsilon(\delta_{u_1v_2}+\delta_{u_2v_1}-\delta_{u_1v_1}-\delta_{u_2v_2})$, where $\varepsilon=\min \{ \pi(u_1,v_2), \pi(u_2,v_1) \}$.
Then $\pi'$ is a simple  transport plan from $\mu_x^p$ to $\mu_y^p$.
Moreover,
\begin{align*}
    \sum\limits_{u,v\in V}\pi(u,v)\rho(u,v)- \sum\limits_{u,v\in V}\pi'(u,v)\rho(u,v)
    &= \varepsilon (\rho(u_1,v_2)+\rho(u_2,v_1)-\rho(u_1,v_1)-\rho(u_2,v_2))
    \\ 
    &\geq \varepsilon(2+2-1-3)=0.
\end{align*}
Since $\pi$ is an optimal transport plan, $\pi'$ is also optimal. 
However, since $\rho(u_1,v_1)= 1$, $\rho(u_1,v_2)\geq 2$, and $\rho(u_2,v_1)\geq 2$, we find
$$\sum_{\rho(u,v)=1}\pi(u,v) -  \sum_{\rho(u,v)=1}\pi'(u,v)\le -\varepsilon< 0,$$
which contradicts the choice of $\pi$. 
\end{proof}

\begin{lemma}\label{four different vertices}
    There do not exist four distinct vertices $u_1,v_1,u_2,v_2\in V(G)$ such that $\rho(u_1,v_1)\geq 3$ and $\rho(u_2,v_2)\geq 3$. 
\end{lemma}
\begin{proof}
    Suppose to the contrary that there exist four distinct vertices $u_1,v_1,u_2,v_2\in V(G)$ such that $\rho(u_1,v_1)\geq 3$ and $\rho(u_2,v_2)\geq 3$.
    It follows that the remaining $n-4$ vertices cannot be adjacent to both $u_1$ and $v_1$, nor to both $u_2$ and $v_2$.
    Together with $\rho(u_1,v_1)\geq 3$ and $\rho(u_2,v_2)\geq 3$, this gives at least $2(n-4)+2$ edges in $\overline{G}$.
    When $n\geq 8$, $2(n-4)+2 > n+\left\lceil \frac{n}{2} \right\rceil-3+\epsilon_n$, which yields a contradiction.
\end{proof}
    For convenience, we denote by $a:=|A_{xy}|$ the size of $A_{xy}$. 
\begin{lemma}\label{edge-in-complement}
    There are exactly $N_x+N_y+2(n-2-a-N_x-N_y)$ edges incident to $x$ or $y$ in $\overline{G}$.
\end{lemma}

\begin{proof}
    We know that the vertices in $N_y$ are adjacent to $x$ in $\overline{G}$, and those in $N_x$ are adjacent to $y$.
    There are $n-2-a-N_x-N_y$ vertices that are adjacent to both $x$ and $y$ in $\overline{G}$.
    The result follows directly.
\end{proof}

We calculate the value of $W(\mu_x^p,\mu_y^p)-1$ as follows.
    \begin{align}\notag
        W(\mu_x^p,\mu_y^p)-1 & =\sum\limits_{u,v\in V} \pi(u,v)(\rho(u,v)-1)\\ 
        & = -\sum\limits_{\rho(u,v)=0}  \pi(u,v) +\sum\limits_{\rho(u,v)\geq 2}  \pi(u,v) (\rho(u,v)-1)      
        \label{**}.
    \end{align}
    Since $\pi$ is simple and $p\in \left[ \frac{1}{1+d_y},1 \right)$, we have
    \begin{align}\label{PIA}
        \sum\limits_{\rho(u,v)=0}  \pi(u,v) =a\cdot\frac{1-p}{d_x}+\frac{1-p}{d_x}+\frac{1-p}{d_y}
    \end{align}
    By the definition of $V_1$, 
    \begin{align*}
        \sum\limits_{\rho(u,v)\geq 2}  \pi(u,v) \rho(u,v)\le |V_1|\cdot\frac{1-p}{d_x}.
    \end{align*}
    Substituting the above inequality and equation \eqref{PIA} into equation \eqref{**} yields
    \begin{align}
        W(\mu_x^p,\mu_y^p)-1 
        \label{***}
        \le -\left(a\cdot\frac{1-p}{d_x}+\frac{1-p}{d_x}+\frac{1-p}{d_y}\right)+ |V_1|\cdot\frac{1-p}{d_x} + \sum\limits_{\rho(u,v)=3}  \pi(u,v).
    \end{align}

\begin{lemma}\label{d(u,v)<3}
For any two vertices $u,v\in V$ with $\pi(u,v)>0$, we have $\rho(u,v)\leq 2$.
\end{lemma}
\begin{proof}
Suppose to the contrary that there exist two vertices $u_1,v_1$ such that $\pi(u_1,v_1)>0$ and $\rho(u_1,v_1)=3$.
Then, $u_1\in N_x$ and $v_1\in N_y$. Hence, $|N_y|\geq 1$, $\rho(u_1,x)=1$, and $\rho(v_1,y)=1$.

\begin{claim}\label{u1v1}
    There are at least $n-3$ edges in $\overline{G}$ that are incident to at least one of $u_1$ and $v_1$, but are not incident to either $x$ or $y$.
\end{claim}
\begin{proof}
    Since $\rho(u_1,v_1)=3$, all the $n-4$ vertices in $V\backslash \{x,y,u_1,v_1\}$) are adjacent to at least one of $u_1$ and $v_1$ in $\overline{G}$.
    Besides, $u_1v_1$ is also an edge in  $\overline{G}$.
    Overall, there are at least $n-3$ edges satisfying the desired property.
\end{proof}

We next estimate the size of $a$.

\begin{claim}\label{a}
    $a \geq \floor{\frac{n}{2}} -2-\epsilon_n$.
\end{claim}
\begin{proof}
    Let us count the number of edges in $\overline{G}$.
    By Lemma~\ref{edge-in-complement}, there are  exactly $N_x+N_y+2(n-2-a-N_x-N_y)$ edges incident to $x$ or $y$ in $\overline{G}$.
    By Claim~\ref{u1v1}, there are at least $n-3$ edges in $\overline{G}$ that are incident to at least one of $u_1$ or $v_1$, but are incident to neither $x$ nor $y$.
    It follows that 
    $$N_x+N_y+2(n-2-a-N_x-N_y)+(n-3)\leq n+\ceil{\frac{n}{2}}-3+\epsilon_n.$$
    Combined with the fact that $n-2-a-N_x-N_y\geq 0$, the above inequality implies
    $$N_x+N_y+(n-2-a-N_x-N_y)+(n-3)\leq n+\left\lceil \frac{n}{2} \right\rceil-3+\epsilon_n,$$
    and the claim follows directly.
\end{proof}

\begin{claim}\label{432}
    $|V_1| \geq a+1$.
\end{claim} 
\begin{proof}
    Recall that $\rho (u_1,v_1)=3$.
    By Lemma~\ref{four different vertices}, if $\rho(u,v)=3$, then $u=u_1$ or $v=v_1$. 
    Further, either $u=u_1$ for every pair $(u,v)$ with $\rho(u,v)=3$, or $v=v_1$ for all such pairs.
    Otherwise, there are two disjoint pairs $u,v$ and $u',v'$ with $\rho(u,v)=\rho(u',v')=3$ such that $u\ne u_1$ and $v'\ne v_1$, which contradicts Lemma~\ref{four different vertices}.
    Therefore, we have
    \begin{align*}
        \sum\limits_{\rho(u,v)=3}  \pi(u,v) \leq \max\{ \mu_x^p(u_1),\mu_y^p(v_1)\}=\frac{1-p}{d_y}.
    \end{align*}
    Substituting this into \eqref{***} yields
    \begin{align*}
        W(\mu_x^p,\mu_y^p)-1 
        \le -\left(a\cdot\frac{1-p}{d_x}+\frac{1-p}{d_x}+\frac{1-p}{d_y}\right)+ |V_1|\cdot\frac{1-p}{d_x} + \frac{1-p}{d_y}.
    \end{align*}
    It follows that
    \begin{align*}
        \kappa_{LLY}(x,y)=\lim_{p\to 1}\frac{1-W(\mu_x^p,\mu_y^p)}{1-p} \ge \frac{a+1}{d_x}-\frac{|V_1|}{d_x}.
    \end{align*}
    The claim then follows from the assumption that $\kappa_{LLY}(x,y)\leq 0$.
\end{proof}

\begin{claim} \label{V_2=v_1}
$V_2=\{v_1\}$.
\end{claim}

\begin{proof}
Suppose to the contrary that there exists a vertex $z_1\in V_2\setminus \{v_1\}$.
Notice that $z_1$ is not adjacent to any vertex in $V_1\setminus\{u_1\}$ by Lemma~\ref{E(V_1,V_2)emptyset}.
It follows from Claim~\ref{432} that there are $|V_1|-1\geq a$ edges between $z_1$ and $V_1\setminus\{u_1\}$ in $\overline{G}$.
These $a$ edges have not been counted in Lemma~\ref{edge-in-complement} or Claim~\ref{u1v1}.
Therefore, there are at least $N_x+N_y+2(n-2-a-N_x-N_y)+(n-3)+a$ edges in $\overline{G}$.
However, for $n\geq 8$, we have
$$N_x+N_y+2(n-2-a-N_x-N_y)+(n-3)+a > n+\left\lceil \frac{n}{2} \right\rceil-3+\epsilon_n,$$
as $n-2-a-N_x-N_y\ge 0$. 
This leads to a contradiction.
\end{proof}

Since $\pi$ is simple and optimal, we obtain that
\begin{align}\notag
    W(\mu_x^p,\mu_y^p) = &\sum\limits_{u,v\in V} \pi(u,v)\rho(u,v) \\ \label{43}
    =&\sum\limits_{u\in V} \pi(u,y)\rho(u,y)+\sum\limits_{u\in V,v\in A_{xy}} \pi(u,v)\rho(u,v)+\sum\limits_{u\in V,v\in N_y} \pi(u,v)\rho(u,v).
\end{align}
By the choice of $\pi$, we have
\begin{align}\label{44}
    \sum\limits_{u\in V} \pi(u,y)\rho(u,y)=\left(p-\frac{1-p}{d_y} \right)+2\left(\frac{1-p}{d_y}-\frac{1-p}{d_x}\right).
\end{align}
According to Claim \ref{V_2=v_1} and the definition of $V_2$, we deduce that
\begin{align*}
    &\sum\limits_{u\in V,v\in A_{xy}} \pi(u,v)\rho(u,v)+\sum\limits_{u\in V,v\in N_y} \pi(u,v)\rho(u,v)\\
    \le\ &\sum\limits_{u\in V,v\in A_{xy}} \pi(u,v)+\sum\limits_{u\in V,v\in N_y\backslash \{v_1 \}} \pi(u,v)+ 3\sum\limits_{u\in V} \pi(u,v_1)\\
    =\ &a\left( \frac{1-p}{d_y}-\frac{1-p}{d_x} \right) + (|N_y|-1)\cdot \frac{1-p}{d_y}+3\cdot\frac{1-p}{d_y}.
\end{align*}
Substituting the above inequality and \eqref{44} into inequality \eqref{43} yields
\begin{align*}
    W(\mu_x^p,\mu_y^p)\le p - (a+2)\cdot \frac{1-p}{d_x}+(3+a+|N_y|)\cdot \frac{1-p}{d_y}.
\end{align*}
Therefore, we derive
\begin{align*}
    \kappa_{LLY}(x,y)=&\lim_{p\to 1}\frac{1- W(\mu_x^p,\mu_y^p)}{1-p}
    \geq  1+  \frac{a+2}{d_x}- \frac{3+a+|N_y|}{d_y}.
\end{align*}
By the fact that $d_y=a+|N_y|+1$, $d_x\le n-|N_y|-1$, and $|N_y|\ge |\{v_1 \}|=1$, we obtain
\begin{align*}
    \kappa_{LLY}(x,y)
    \ge \frac{a+2}{d_x}-\frac{2}{d_y}
    \geq & \frac{a+2}{n-|N_y|-1}-\frac{2}{a+|N_y|+1}\\
    =& \frac{a^2+3a+4+(4+a)|N_y|-2n}{(n-|N_y|-1)(a+|N_y|+1)}\\
    \geq & \frac{a^2+4a+8-2n}{(n-|N_y|-1)(a+|N_y|+1)}.
\end{align*}
By Claim~\ref{a}, the above expression is positive when $n\ge 8$, which is a contradiction.
This completes the proof of Lemma \ref{d(u,v)<3}.
\end{proof}

Set $V_3\coloneqq \{w\in N_x\setminus V_1|\ \forall z\in V_2,\ \rho(w,z)\geq 2 \}$, $V_4\coloneqq N_x\setminus(V_1\cup V_3)$, and $V_5\coloneqq \{ w\in N_y\cup A_{xy}|\ \exists z\in V_4,\ \pi(z,w)>0\}\setminus V_2$.
Thus, $|V_1|+|V_3|+|V_4|=|N_x|$ and $E(V_3,V_2)=\emptyset$.

\begin{lemma}\label{E(V_1,V_5)=empty}
    $E(V_1,V_5)=\emptyset$.
\end{lemma}
\begin{proof}
Suppose to the contrary that there exist $z_1\in V_1$ and $z_2\in V_5$ such that $z_1z_2\in E(G)$.
By the definition of $V_5$, there is a vertex $z_3\in V_4$ such that $\pi(z_3,z_2)>0$.
By the definition of $V_4$, there exists a vertex $z_4\in V_2$ such that $z_3z_4\in E(G)$.
By the definition of $V_1$ and $V_2$, there exists a vertex $z_5\in V_1$ such that $\pi(z_5,z_4)>0$ and $\rho(z_5,z_4)\geq 2$.
Moreover, there is a vertex $z_6\in V_2\cup \{y \}$ such that $\pi(z_1,z_6)>0$ and  $\rho(z_1,z_6)\geq 2$.
Here $z_1$ and $z_5$ may be the same vertex while $z_4$ and $z_6$ may also be the same vertex.
Set 
$$\pi'\coloneqq\pi-\varepsilon(\delta_{z_3z_2}+\delta_{z_5z_4}+\delta_{z_1z_6}-\delta_{z_3z_4}-\delta_{z_5z_6}-\delta_{z_1z_2}),$$
 where $\varepsilon=\min \{ \pi(z_3,z_2), \pi(z_5,z_4), \pi(z_1,z_6) \}> 0$.
Then $\pi'$ is a simple  transport plan from $\mu_x^p$ to $\mu_y^p$.
It follows that
\begin{align*}
    \sum\limits_{u,v\in V}\pi(u,v)\rho(u,v)- \sum\limits_{u,v\in V}\pi'(u,v)\rho(u,v)\geq \varepsilon(1+2+2-1-3-1)=0.
\end{align*}
Since $\pi$ is an optimal transport plan, $\pi'$ is also optimal. 
However, since $\rho(z_3,z_4)= 1$, $\rho(z_1,z_2)=1$, $\rho(z_5,z_4)\geq 2$,  and $\rho(z_1,z_6)\geq 2$, we find
$$\sum_{\rho(u,v)=1}\pi(u,v) -  \sum_{\rho(u,v)=1}\pi'(u,v)\le \varepsilon(1-1-1)< 0,$$
which contradicts the choice of $\pi$. 
\end{proof}

Overall, by Lemma~\ref{E(V_1,V_2)emptyset}, Lemma~\ref{edge-in-complement}, and Lemma~\ref{E(V_1,V_5)=empty}, there are at least $N_x+N_y+2(n-2-a-N_x-N_y)+|V_2||V_3|+|V_1||V_2|+|V_1||V_5|$ edges in $\overline{G}$.
    It follows that 
    $$N_x+N_y+2(n-2-a-N_x-N_y)+|V_2||V_3|+|V_1||V_2|+|V_1||V_5|\leq n+\left\lceil \frac{n}{2} \right\rceil-3+\epsilon_n,$$
which is equivalent to
\begin{align}\label{final}
\floor{\frac n2} -N_x-N_y-2a +|V_2|(|V_1|+|V_3|) +|V_5||V_1|\leq 1+\epsilon_n.
\end{align}

\begin{lemma}\label{|V_1|}
    $|V_1|\geq 1+a+\frac{d_x}{d_y}$.
\end{lemma}

\begin{proof}
    By Lemma~\ref{d(u,v)<3} and inequality~\eqref{***}, we have
    \begin{align}\notag
        W(\mu_x^p,\mu_y^p)-1 & \le -\left(a\cdot\frac{1-p}{d_x}+\frac{1-p}{d_x}+\frac{1-p}{d_y}\right)+ |V_1|\cdot\frac{1-p}{d_x}.
    \end{align}
    It follows that 
\begin{align*}
    \kappa_{LLY}(x,y)=&\lim_{p\to 1}\frac{1- W(\mu_x^p,\mu_y^p)}{1-p}
    \geq \frac{a+1-|V_1|}{d_x}+ \frac{1}{d_y}.
\end{align*}
    The desired result then follows from the assumption that
    $\kappa_{LLY}(x,y)\leq 0$.
\end{proof}

For simplicity, set $M_0\coloneqq \sum\limits_{u\in V_4, v\in V_2}\pi(u,v)$. Let $\delta \coloneqq 0$ if $|N_y|\ge 1$, and $\delta \coloneqq 1$ if $|N_y|= 0$.

\begin{lemma}\label{|V_2|}
    We have $d_x-\delta d_y>0$ and $|V_2|\geq \frac{d_y(a+2)}{d_x-\delta d_y}+\frac{d_xd_yM_0}{(d_x-\delta d_y)(1-p)}$.
\end{lemma}

\begin{proof}
    We first prove $d_x-\delta d_y>0$.
    Recall that $d_x\ge d_y$. If $\delta =0$, then $d_x-\delta d_y=d_x>0$. If $\delta =1$, then $N_y=\emptyset$ and $d_x-\delta d_y=d_x-d_y$. We claim that $d_x\ne d_y$. Otherwise, we have $N_x=N_y=\emptyset$, and hence $a=d_x-1$. Then Lemma \ref{triangle} gives
    $$\kappa_{LLY}(x, y)\ge -2+\frac{4}{d_x}+ \frac{3a}{d_x} = \frac{d_x+1}{d_x}>0,$$
    which is a contradiction. Thus, $d_x-\delta d_y>0$.

    We next prove the lower bound on $|V_2|$. Let 
    $$m_2\coloneqq\sum_{u\in N_x}\ \sum_{\substack{v\in N_y\cup A_{xy}\\\rho(u,v)=2}} \pi(u,v).$$
    By Lemma~\ref{d(u,v)<3}, we deduce that
    \begin{align}\notag
        W(\mu_x^p,\mu_y^p)=&\sum_{u,v\in V}\pi(u,v)\rho(u,v) \\ \notag
        =& \sum_{u\in V}\pi(u,y)\rho(u,y)+ \sum_{u\in N_x}\ \sum_{\substack{v\in N_y\cup A_{xy}\\\rho(u,v)=1}} \pi(u,v)+ 2m_2\\ \label{pdy}
        =&\left(p-\frac{1-p}{d_y}\right)+ 2 \left(\frac{1-p}{d_y}-\frac{1-p}{d_x}\right)+\sum_{u\in N_x}\ \sum_{\substack{v\in N_y\cup A_{xy}\\\rho(u,v)=1}} \pi(u,v)
        +2m_2.
    \end{align}
    Observe that
    \begin{align}\notag
    \sum_{u\in N_x}\ \sum_{\substack{v\in N_y\cup A_{xy}\\ \rho(u,v)=1}} \pi(u,v)+m_2 
    &= \sum_{u\in N_x}\ \sum_{v\in V} \pi(u,v)- \sum_{u\in N_x} \pi(u,y)\\ \label{Nx1p}
    &= |N_x|\cdot\frac{1-p}{d_x}-\left(\frac{1-p}{d_y}-\frac{1-p}{d_x}\right).
    \end{align}
    By the definition of $V_1$, we have $\rho (u,v)=1$ for any $u\in V_4$ and $v\in V_2$ with $\pi (u,v)>0$.
    By the definition of $V_2$, we have $v\in V_2$ for any $u\in N_x$ and $v\in N_y\cup A_{xy}$ with $\pi (u,v)>0$ and $\rho(u,v)=2$.
    Hence, by the definition of $m_2$ and $M_0$, we deduce that
    \begin{align*}
        m_2+ M_0+\sum_{v\in V_2}\pi(v,v)\le \sum_{u\in V}\sum_{v\in V_2}\pi(u,v) \le \sum_{v\in V_2}\mu^p_y (v) =|V_2|\cdot\frac{1-p}{d_y}.
    \end{align*}
    Note that if $N_y=\emptyset$, then $V_2\subseteq A_{xy}$. Since $\pi$ is simple, we find
    \begin{align*}
        m_2+ M_0\le |V_2|\cdot\frac{1-p}{d_y}-\sum_{v\in V_2}\pi(v,v)\le |V_2|(1-p)\left(\frac{1}{d_y}- \frac{\delta}{d_x} \right).
    \end{align*}
    Combining equations \eqref{pdy} and \eqref{Nx1p} and the above inequality, we derive
    \begin{align*}
        W(\mu_x^p,\mu_y^p)\le p-\frac{1-p}{d_x}+|N_x|\cdot\frac{1-p}{d_x} +|V_2|(1-p)\left(\frac{1}{d_y}- \frac{\delta}{d_x} \right)- M_0.
    \end{align*}
    Then equation \eqref{eq:bourne} yields
\begin{align}\notag
    \kappa_{LLY}(x,y)=\frac{1- W(\mu_x^p,\mu_y^p)}{1-p}&\ge 1+\frac{1}{d_x}-\frac{|N_x|}{d_x} -|V_2|\left(\frac{1}{d_y}- \frac{\delta}{d_x} \right)+ \frac{M_0}{1-p}\\ \notag
    &=\frac{a+2}{d_x} -|V_2|\left(\frac{1}{d_y}- \frac{\delta}{d_x} \right)+ \frac{M_0}{1-p}.
\end{align}
    The desired result then follows by the assumption that $\kappa_{LLY}(x,y)\le 0$
\end{proof}

\begin{lemma}\label{|V_5|}
$|V_5|\geq \frac{d_y|V_4|}{d_x-\delta d_y}-\frac{d_xd_yM_0}{(d_x-\delta d_y)(1-p)}$.
\end{lemma}

\begin{proof}
    By the definition of $V_1$, if $\pi(u,y)>0$ for $u\in N_x$, then $u\in V_1$.
    Thus, by the definition of $V_5$ and $M_0$, we arrive at
    \begin{equation}\label{v41p}
        |V_4|\cdot\frac{1-p}{d_x}=\sum_{u\in V_4}\sum_{v\in V}\pi(u,v) = M_0+ \sum_{u\in V_4}\sum_{v\in V_5}\pi(u,v).
    \end{equation}
    Note that if $N_y=\emptyset$, then $V_5\subseteq A_{xy}$. Since $\pi$ is simple, we find
    $$\sum_{u\in V_4}\sum_{v\in V_5}\pi(u,v)\le |V_5|\cdot\frac{1-p}{d_y}-\sum_{v\in V_5}\pi(v,v) \le |V_5|(1-p)\left(\frac{1}{d_y}- \frac{\delta}{d_x} \right).$$
    Combining the above inequality with equation \eqref{v41p} yields the desired bound on $|V_5|$.
\end{proof}

By the definition of $V_4$, if $a=|N_y|=0$, then $|V_4|=0$.
By Lemma~\ref{|V_1|}, Lemma~\ref{|V_2|}, Lemma~\ref{|V_5|}, and the fact that $d_x=1+a+|N_x|$, $d_y=1+a+|N_y|$, $|N_x|\geq |N_y|$, $|N_x|+|N_y|+a+2\leq n$, and $|V_1|+|V_3|+|V_4|=|N_x|$, we can apply Lemma~\ref{computation} with $M=\frac{d_xd_yM_0}{(d_x-\delta d_y)(1-p)}$ to derive 
\begin{align}\label{fin}
    \floor{\frac{n}{2}}-|N_x|-|N_y|-2a+|V_2|(|V_1|+|V_3|)+|V_5||V_1|\geq 2.
\end{align}
Combined with inequality \eqref{final}, we deduce that $\epsilon_n= 1$ and that inequality \eqref{fin} takes equal.
Then by Lemma~\ref{computation}, we have $|V_1|=\left\lceil\frac{n}{2}\right\rceil$, $|V_2|=1$, and $|V_4|=\left\lfloor\frac{n}{2}\right\rfloor-3$. Set $V_2=\{w\}$.
Then $V=\{x,y,w\}\cup V_1 \cup V_4$.
By Lemma~\ref{E(V_1,V_2)emptyset}, the number of edges in $\overline{G}$ is at least $$|V_2|+(|V_1|+|V_4|)+|V_1||V_2|=n+\ceil{\frac{n}{2}}-2.$$
Thus, $E(\overline{G})=\{xw\}\cup \{yz\ |\ z\in V_1\cup V_4\}\cup \{z_1z_2\ |\ z_1\in V_1, z_2\in V_2\}$, and the graph $G$ is uniquely determined. Next, we show that the curvature of $xy$ is non-positive, and hence $G$ is indeed an extremal graph.
A schematic of $G$ is depicted in Figure \ref{fig:enter-label}.
\begin{figure}[htbp]
    \centering    \begin{tikzpicture}
    \tikzset{
        big circle/.style={circle, draw, minimum size=1.4cm, inner sep=0pt, font=\Large},
        dot/.style={circle, fill=black, minimum size=4pt, inner sep=0pt}
    }

    \node[big circle] (Ka) at (-2.2, 0) {$K_{\left\lceil\frac{n}{2}\right\rceil}$};
    \node[big circle] (Kb) at (0, 0) {$K_{\left\lfloor\frac{n}{2}\right\rfloor-3}$};
    \node[dot, label=right:{\Large $w$}] (w) at (2.2, 0) {};

    \node[dot, label=left:{\Large $x$}] (x) at (-1.1, 2) {};
    \node[dot, label=right:{\Large $y$}] (y) at (1.6, 2) {};

    \draw (Ka) -- (Kb);
    \draw (Kb) -- (w);

    \draw (x) -- (y);

    \draw (x) -- (Ka);
    \draw (x) -- (Kb);
    \draw (y) -- (w);

\end{tikzpicture}
    \caption{A schematic of the graph $G$.}
    \label{fig:enter-label}
\end{figure}
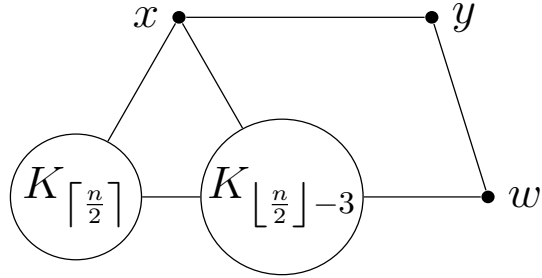

Consider the function $f:N(x)\cup N(y)\rightarrow\mathbb{Z}$ given by
$$f(z)= 
\begin{cases}
-1, & \text { if } z\in A;\\
0, & \text { if } z\in \{x\} \cup B; \\
1, & \text { if } z\in \{y,w\}.
\end{cases}$$
Then, $f(y)-f(x)=1$ and $f\in Lip(1)$.
By Theorem~\ref{Curvature via the Laplacian}, we have 
\[\kappa_{LLY}(x,y)\leq \Delta f(x)-\Delta f(y)=\frac{1}{n-2}\left(1-\ceil{\frac{n}{2}}\right)-\frac{1}{2}(1+0-2)\leq 0.\]
This completes the proof. 

\begin{remark}
    Note that the graph depicted in Figure \ref{fig:enter-label} still satisfies $\kappa_{LLY}(x,y)\le 0$ when $n$ is odd.
    Moreover, by a more detailed analysis in the proof of Lemma~\ref{computation}, one can check that, for odd $n$, there are two additional sets of parameters for which inequality~\eqref{inequality} becomes an equality.
    Namely, $V_1=\frac{n-1}{2}, a=0,V_2=N_y=1,V_4=\frac{n-7}{2},V_3=V_5=0$; and $V_1=\frac{n+3}{2}, a=V_2=1,V_4=\frac{n-9}{2},V_3=V_5=N_y=0$.
    Similarly, each set of parameters determines $E(\overline{G})$ and hence determines the graph $G$.
    In the next section, we will verify that the resulting graphs are both extremal graphs. Therefore, for odd $n\ge 13$, Theorem~\ref{main} indeed has exactly three extremal graphs.
\end{remark}

\section{Sharpness discussions}\label{section:sharpness}

For odd $n\ge 9$, we present two additional extremal graphs $G_1$ and $G_2$ of order $n$, distinct from the graph $G$ depicted in Figure \ref{fig:enter-label}. 
Each has exactly $\frac{n^2-3n}{2}-\left\lceil\frac{n}{2}\right\rceil+2$ edges and at least one edge of non-positive curvature.
This in particular illustrates that the condition that $n$ is even in Theorem \ref{main} is necessary.

Let $G_1=(V,E)$ be the graph  of order $n$ defined as follows.
The vertex set $V\coloneqq\{x,y,w,u\}\sqcup A\sqcup B$, where $|A|=\frac{n-1}{2}$ and $|B|=\frac{n-7}{2}$.
The edge set 
$E\coloneqq\{xy\}\sqcup \{yw\}\sqcup E_1\sqcup E_2\sqcup E_3,$ where $E_1\coloneqq\{uv\ |\ u, v\in A\cup B\cup \{x \} \text{ and } u\neq v\}$, $E_2\coloneqq\{wv\ |\ v\in B\}$, and $E_3\coloneqq\{uv\ |\ v\in A\cup B\cup\{w\}\}$.
Then, $G_1$ has exactly $\frac{n^2-3n}{2}-\left\lceil\frac{n}{2}\right\rceil+2$ edges.
A schematic of $G_1$ is depicted in Figure~\ref{G_1}.
\begin{figure}[htbp]
    \centering     \begin{tikzpicture}

  \tikzset{
        big circle/.style={circle, draw, minimum size=2.2cm, inner sep=0pt, font=\Large},
        dot/.style={circle, fill=black, minimum size=4pt, inner sep=0pt}
    }

\node[big circle, font=\huge] (Ka) at (-1.25,0) {$K_{\frac{n-1}{2}}$};
\node[big circle, font=\huge] (Kb) at (1.90,0) {$K_{\frac{n-7}{2}}$};

\node[dot, label=left:{\Large $x$}] (x) at (0,2.55) {};
\node[dot, label=right:{\Large $y$}] (y) at (3.65,2.55) {};
\node[dot, label=right:{\Large $w$}] (w) at (4.85,0) {};
\node[dot, label=below:{\Large $u$}] (u) at (1.90,-2.45) {};

\draw (x) -- (y);
\draw (y) -- (w);
\draw (w) -- (u);

\draw (Ka) -- (Kb);
\draw (Ka) -- (x);
\draw (Kb) -- (x);
\draw (Ka) -- (u);
\draw (Kb) -- (u);
\draw (Kb) -- (w);

\end{tikzpicture}
    \caption{A schematic of the graph $G_1$.}
    \label{G_1}
\end{figure}
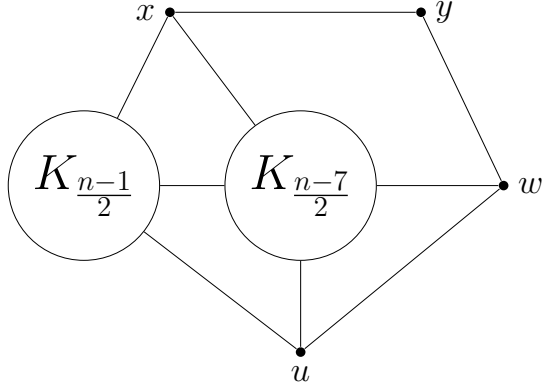

We next estimate the Lin--Lu--Yau curvature $\kappa_{LLY}(x,y)$ of $xy$.
Consider the function $f_1:N(x)\cup N(y)\rightarrow\mathbb{Z}$ given by
$$f_1(z)= 
\begin{cases}
-1, & \text { if } z\in A;\\
0, & \text { if } z\in \{x\} \cup B; \\
1, & \text { if } z\in \{y,w\}.
\end{cases}$$
Then, $f_1(y)-f_1(x)=1$ and $f_1\in Lip(1)$.
By Theorem~\ref{Curvature via the Laplacian}, we have 
\[\kappa_{LLY}(x,y)\leq \frac{1}{n-3}\left(1-\frac{n-1}{2}\right)-\frac{1}{2}(1+0-2)= 0.\]

Let $G_2=(V,E)$ be the graph of order $n$ defined as follows.
The vertex set $V\coloneqq\{x,y,w\}\sqcup A\sqcup B$, where $|A|=\frac{n+3}{2}$ and $|B|=\frac{n-9}{2}$.
The edge set $E\coloneqq\{xy\}\sqcup \{yw\}\sqcup E_1\sqcup E_2$, where $E_1\coloneqq\{uv\ |\ u, v\in A\cup B\cup \{x \} \text{ and } u\neq v\}$ and $E_2\coloneqq\{wv\ |\ v\in B\cup\{x\}\}$.
Then, $G$ has exactly $\frac{n^2-3n}{2}-\left\lceil\frac{n}{2}\right\rceil+2$ edges.
A schematic of $G_2$ is depicted in Figure~\ref{G_2}.
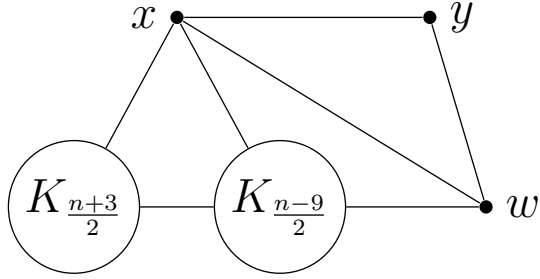
\begin{figure}[htbp]
    \centering    \begin{tikzpicture}
    \tikzset{
        big circle/.style={circle, draw, minimum size=1.4cm, inner sep=0pt, font=\Large},
        dot/.style={circle, fill=black, minimum size=4pt, inner sep=0pt}
    }

    \node[big circle] (Ka) at (-2.2, 0) {$K_{\frac{n+3}{2}}$};
    \node[big circle] (Kb) at (0, 0) {$K_{\frac{n-9}{2}}$};
    \node[dot, label=right:{\Large $w$}] (w) at (2.2, 0) {};

    \node[dot, label=left:{\Large $x$}] (x) at (-1.1, 2) {};
    \node[dot, label=right:{\Large $y$}] (y) at (1.6, 2) {};

    \draw (Ka) -- (Kb);
    \draw (Kb) -- (w);

    \draw (x) -- (y);

    \draw (x) -- (Ka);
    \draw (x) -- (Kb);
    \draw (y) -- (w);
\draw (x) -- (w);
\end{tikzpicture}
    \caption{A schematic of the graph $G_2$.}
    \label{G_2}
\end{figure}
Similarly, consider the function $f_2:N(x)\cup N(y)\rightarrow\mathbb{Z}$ given by
$$f_2(z)= 
\begin{cases}
-1, & \text { if } z\in A;\\
0, & \text { if } z\in \{x\} \cup B; \\
1, & \text { if } z\in \{y,w\}.
\end{cases}$$
Then, Theorem~\ref{Curvature via the Laplacian} yields
\[\kappa_{LLY}(x,y)\leq \frac{1}{n-1}\left(2-\frac{n+3}{2}\right)-\frac{1}{2}(1+0-2)= 0.\]

Finally, we present an example for $7\leq n \leq 11$ to illustrate that the conditions $n\geq 8$ and $n\geq 12$ in Theorem~\ref{main} are both necessary.

Let $H=(V,E)$ be the graph of order $n$ defined as follows.
The vertex set $V\coloneqq\{x,y\}\sqcup A\sqcup B$, where $|A|=a$, $|B|=b$, and $a+b+2=n$.
The edge set $E\coloneqq\{xy\}\sqcup E_1\sqcup E_2$, where $E_1\coloneqq\{uv\ |\ u, v\in A\cup B \text{ and } u\neq v\}$ and $E_2\coloneqq\{xv\ |\ v\in B\}$.
A schematic of $H$ is depicted in Figure~\ref{H}.

\begin{figure}[htbp]
    \centering    \begin{tikzpicture}[
    scale=1,
    every node/.style={font=\Large},
    point/.style={circle, fill=black, inner sep=2pt}
]

\coordinate (A) at (0,0);
\coordinate (B) at (2.2,0);
\coordinate (P) at (4.2,0);
\coordinate (Q) at (5.4,0);

\draw (A) -- (B);
\draw (B) -- (Q);

\node[
    circle,
    draw,
    fill=white,
    minimum size=1.8cm,
    inner sep=0pt, font=\huge
] at (A) {$K_a$};

\node[
    circle,
    draw,
    fill=white,
    minimum size=1.8cm,
    inner sep=0pt, font=\huge
] at (B) {$K_b$};

\node[point, label=below:{\Large $x$}] at (P) {};
\node[point, label=below:{\Large $y$}] at (Q) {};

\end{tikzpicture}
    \caption{A schematic of the graph $H$.}
    \label{H}
\end{figure}
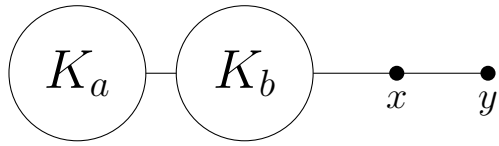

For $n=8,9,10,11$, select $a=\ceil{\frac{n}{2}}$ and $b=\floor{\frac{n}{2}}-2$.
Then $H$ has exactly $\frac{n^2-3n}{2}-\left\lceil\frac{n}{2}\right\rceil+2$ edges.
It is direct to check that $H$ has some edges with non-positive Lin--Lu--Yau curvature by using the graph curvature calculator~\cite{CKLLS22}, which is a freely accessible interactive app at 
https://www.mas.ncl.ac.uk/graph-curvature/.
This illustrates that the condition $n\geq 12$ in Theorem~\ref{main} is necessary.

For $n=7$, select $a=3$ and $b=2$.
Then $H$ has exactly $\frac{n^2-3n}{2}-\left\lceil\frac{n}{2}\right\rceil+3=13$ edges.
Similarly, one can check that $H$ has some edges with negative Lin--Lu--Yau curvature by using the graph curvature calculator~\cite{CKLLS22}, as shown in Figure~\ref{n=7}.
This illustrates that the condition $n\geq 8$ in Theorem~\ref{main} is also necessary.




\begin{figure}[htbp]
    \centering
\includegraphics[width=0.7\textwidth,height=0.3\textwidth]{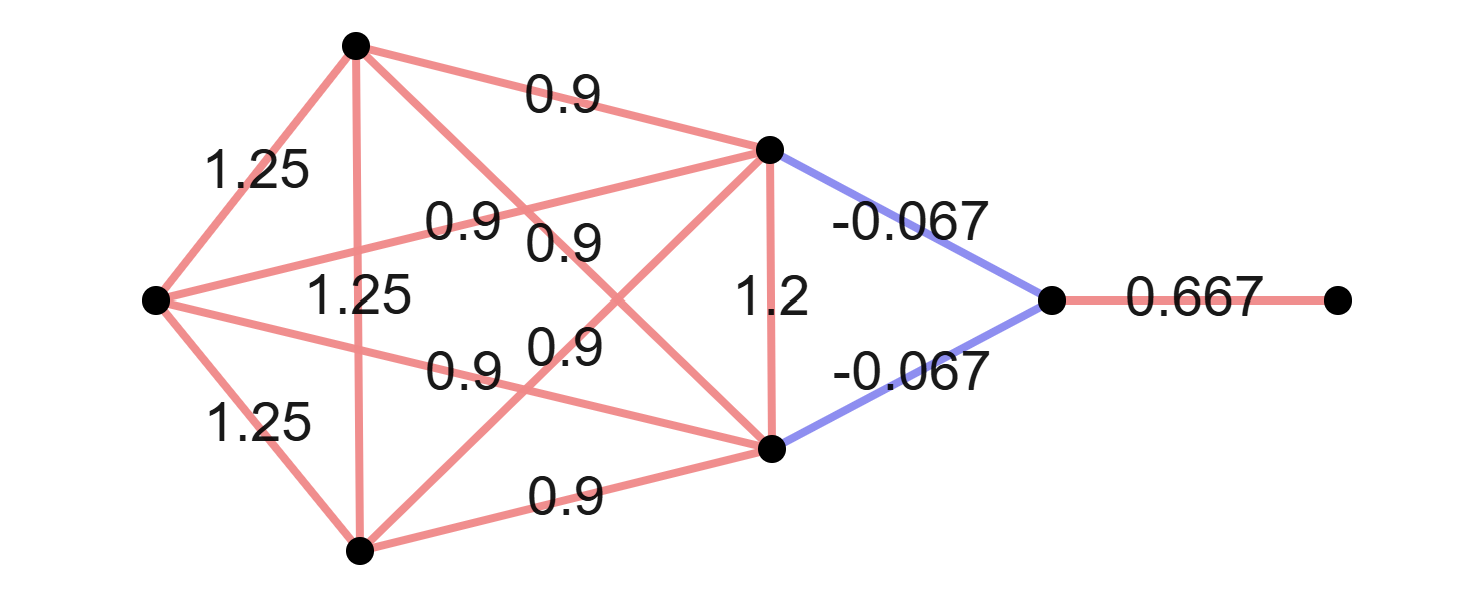}
    \caption{A graph of order 7 containing edges of negative LLY curvature.}
    \label{n=7}
\end{figure}

\section{Outlook}\label{section:extremal-problems}
Theorem~\ref{main} determines the sharp edge threshold above which all edges have positive Ollivier/Lin--Lu--Yau curvature.
In this section, we conclude by formulating several open problems suggested by Theorem~\ref{main}.

\begin{problem}[Curvature extremal problem 1]
Let $n$ be a positive integer and $K$ be a real number.
Determine the minimum $\mathrm{ex}(n,K)$ such that every graph of order $n$ with more than $\mathrm{ex}(n,K)$ edges has Lin--Lu--Yau curvature larger than $K$ on each edge.
\end{problem}

Under the above notion, Theorem~\ref{main} gives $\mathrm{ex}(n,0)=\frac{n^2-3n}{2}-\left\lceil\frac{n}{2}\right\rceil+2$. It might also be interesting to allow certain edges with non-positive curvature.

\begin{problem}[Curvature extremal problem 2]
Let $n$ and $m$ be two positive integers.
Determine the minimum $\mathrm{ex}(n,m)$ such that every graph of order $n$ with more than $\mathrm{ex}(n,m)$ edges has positive Lin--Lu--Yau curvature on all but at most $m$ edges.
\end{problem}

It is also natural to consider other extremal problems, such as the saturation problems.

\begin{problem}[Curvature saturation problem]
Let \(n\) be a positive integer. Determine the minimum number
\(\operatorname{sat}(n)\) of edges in a graph \(G\) of order
\(n\) with the following properties:
\(G\) has at least one edge of non-positive Lin--Lu--Yau curvature, but, for
every non-edge \(uv\), the graph \(G+uv\) has positive Lin--Lu--Yau curvature.
\end{problem}

One can also consider curvature notions other than Ollivier/Lin--Lu--Yau curvature, such as Bakry--\'Emery curvature~\cite{Bakry, LY10} and resistance curvature~\cite{resistance}.



\medskip

\

\noindent\textbf{AI assistance statement:} 
The authors used AI-assisted tools in the proof of Lemma~\ref{computation}.
All other mathematical ideas and proofs were developed independently by the authors without the use of AI-assisted tools.

\section*{Acknowledgement}
\noindent K.C. is grateful to Tian Wu for helpful discussions regarding Lemma~\ref{computation}.
This work is supported by the National Key R\&D Program of China 2023YFA1010200. K.C.'s research is supported by the National Natural Science Foundation of China No. 125B1009. S.L.'s research is supported by the Scientific Research Innovation Capability Support Project for Young Faculty SRICSPYF-ZY2025160 and the National Natural Science Foundation of China No. 12431004.

\end{document}